\documentclass[12pt, reqno]{amsart}
\usepackage{}
\usepackage{amsmath, amsthm, amscd, amsfonts, amssymb, graphicx, color}
\usepackage[bookmarksnumbered, colorlinks, plainpages]{hyperref}
\hypersetup{colorlinks=true,linkcolor=red, anchorcolor=green, citecolor=cyan, urlcolor=red, filecolor=magenta, pdftoolbar=true}

\newtheorem{theorem}{Theorem}[section]
\newtheorem{lemma}[theorem]{Lemma}
\newtheorem{proposition}[theorem]{Proposition}
\newtheorem{corollary}[theorem]{Corollary}
\theoremstyle{definition}
\newtheorem{definition}[theorem]{Definition}

\theoremstyle{remark}

\numberwithin{equation}{section}

\begin{document}
\setcounter{page}{1}

\title[Carleson measure spaces with variable exponents]
{Carleson measure spaces with variable exponents and their applications}

\author[J. Tan]{Jian Tan}

\address{College of Science, Nanjing University of Posts and Telecommunications, Nanjing 210023, China.}
\email{\textcolor[rgb]{0.00,0.00,0.84}{tanjian89@126.com; tj@njupt.edu.cn};
}

\subjclass[2010]{42B25, 42B35, 46E30.}

\keywords{Carleson measure spaces, variable exponents, dual spaces, singular integrals}

\begin{abstract}
In this paper, we introduce the Carleson measure spaces with variable exponents $CMO^{p(\cdot)}$.
By using discrete Littlewood$-$Paley$-$Stein analysis as well as
Frazier and Jawerth's $\varphi-$transform in the variable exponent settings, we show that the dual space of the variable Hardy space $H^{p(\cdot)}$ is $CMO^{p(\cdot)}$.
As applications, we obtain
that Carleson measure spaces with variable exponents $CMO^{p(\cdot)}$, Campanato space with variable exponent
$\mathfrak{L}_{q,p(\cdot),d}$ and H\"older-Zygmund spaces with variable exponents $\mathcal {\dot{H}}_d^{p(\cdot)}$ coincide as sets and the corresponding norms are equivalent. Via
using an argument of weak
density property, we also prove the boundedness
of
Calder\'{o}n-Zygmund singular integral operator acting
on $CMO^{p(\cdot)}$.
\end{abstract} \maketitle

\section{Introduction}
The Hardy and $BMO$ spaces have been
playing a crucial role in modern harmonic analysis since
the early groundbreaking work in Hardy space theory
came from Coifman, Fefferman, Stein and Weiss
in \cite{C,CWe,FS}. In \cite{FS}, Fefferman and Stein obtained
that the space of functions of bounded mean oscillation, $BMO$,
is the dual space of the Hardy space $H^1$ and that the space
$BMO$ can be characterized by Carleson measure, which suggests
that one could use the generalized Carleson measure to characterize
the dual of the Hardy space. For this purpose, we introduce the
Carleson measure spaces $CMO^{p(\cdot)}$ that generalize $CMO^p$
and $BMO$. Note that Carleson measure spaces $CMO^p$ have been studied
in \cite{HLL,LL,LLL,L,LW}.
However, the theory of
Carleson measure spaces is
still unknown in the variable exponent settings.

The main goal of this paper is to develop a complete theory
for the dual spaces of variable Hardy spaces $H^{p(\cdot)}$.
Before stating our main results, we begin with the definition
of Lebesgue and Hardy spaces with variable exponent and some
notations.

For any Lebesgue measurable function $p(\cdot):
\mathbb R^n\rightarrow (0,\infty]$ and for any
measurable subset $E\subset \mathbb{R}^n$, we denote
$p^-(E)= \inf_{x\in E}p(x)$ and $p^+(E)= \sup_{x\in E}p(x).$
Especially, we denote $p^-=p^{-}(\mathbb{R}^n)$ and $p^+=p^{+}(\mathbb{R}^n)$.
Let $p(\cdot)$: $\mathbb{R}^n\rightarrow(0,\infty)$ be a measurable
function with $0<p^-\leq p^+ <\infty$ and $\mathcal{P}^0$
be the set of all these $p(\cdot)$.

\begin{definition}\cite{CF,DHHR,KR}\label{s1de1}\quad
Let $p(\cdot):\mathbb R^n\rightarrow (0,\infty]$
be a Lebesgue measurable function.
The variable Lebesgue space $L^{p(\cdot)}$ consisits of all
Lebesgue measurable functions $f$, for which the quantity
$\int_{\mathbb{R}^n}|\varepsilon f(x)|^{p(x)}dx$ is finite for some
$\varepsilon>0$ and
$$\|f\|_{L^{p(\cdot)}}=\inf{\left\{\lambda>0: \int_{\mathbb{R}^n}\left(\frac{|f(x)|}{\lambda}\right)^{p(x)}dx\leq 1 \right\}}.$$
\end{definition}

We also recall the following class of
exponent function, which can be found in \cite{D}.
Let $\mathcal{B}$ be the set of $p(\cdot)\in \mathcal{P}$ such that the
Hardy-littlewood maximal operator $M$ is bounded on  $L^{p(\cdot)}$.
An important subset of $\mathcal{B}$ is $LH$
condition.

In the study of variable exponent function spaces it is common
to assume that the exponent function $p(\cdot)$ satisfies $LH$
condition.
We say that $p(\cdot)\in LH$, if $p(\cdot)$ satisfies

 $$|p(x)-p(y)|\leq \frac{C}{-\log(|x-y|)} ,\quad |x-y| \leq 1/2$$
and
 $$|p(x)-p(y)|\leq \frac{C}{\log|x|+e} ,\quad |y|\geq |x|.$$

It is well known
that $p(\cdot)\in \mathcal{B}$ if $p(\cdot)\in \mathcal{P}\cap LH.$
Denote by $\mathcal S=\mathcal S(\mathbb R^n)$ the collection
of rapidly decreasing $C^{\infty}$ function on $\mathbb R^n$.
Also, denote by $\mathcal S_\infty$ the functions $f\in\mathcal S$ satisfying
$\int_{\mathbb R^n}f(x)x^\alpha dx=0$ for all muti-indices $\alpha\in \mathbb Z_+^n:=(\{0,1,2,\cdots\})^n$
and $\mathcal S'_\infty$ its topological dual space.
For $f\in \mathcal S'_\infty$,
we recall the definition of the Littlewood-Paley-Stein square function
\begin{align*}
\mathcal{G}(f)(x):=\left(\sum_{j\in \mathbb Z}|\psi_j\ast f(x)|^2\right)^{1/2},
\end{align*}

and the discrete Littlewood-Paley-Stein square function
\begin{align*}
\mathcal{G}^d(f)(x):=\left(\sum_{j\in \mathbb Z}
\sum_{\mathbf k\in\mathbb Z^n}|\psi_j\ast f(2^{-j}\mathbf k)|^2\chi_Q(x)\right)^{1/2},
\end{align*}
where $Q$ denote dyadic cubes in $\mathbb R^n$ with side-lengths $2^{-j}$ and the
lower left-corners of $Q$ are $2^{-j}\mathbf k$.
We also
recall the definition of variable Hardy spaces ${H}^{p(\cdot)}$ as follows.

\begin{definition}\label{s1de2}~(\cite{CW,NS})\quad
Let $f\in \mathcal{S'}$,
$\psi\in \mathcal S$, $p(\cdot)\in {\mathcal{P}^0}$
and $\psi_t(x)=t^{-n}\psi(t^{-1}x)$, $x\in \mathbb{R}^n$.
Denote by $\mathcal{M}$ the grand maximal operator given by
$\mathcal{M}f(x)= \sup\{|\psi_t\ast f(x)|: t>0,\psi \in \mathcal{F}_N\}$ for any fixed large integer $N$,
where $\mathcal{F}_N=\{\varphi \in \mathcal{S}:\int\varphi(x)dx=1,\sum_{|\alpha|\leq N}\sup(1+|x|)^N|\partial ^\alpha \varphi(x)|\leq 1\}$.
The variable Hardy space ${H}^{p(\cdot)}$ is the set of all $f\in \mathcal{S}^\prime$, for which the quantity
$$\|f\|_{{H}^{p(\cdot)}}=\|\mathcal{M}f\|_{{L}^{p(\cdot)}}<\infty.$$
\end{definition}

Throughout this paper, $C$ or $c$ denotes a positive constant that may vary at each occurrence
but is independent to the main parameter, and $A\sim B$ means that there are constants
$C_1>0$ and $C_2>0$ independent of the the main parameter such that $C_1B\leq A\leq C_2B$.
Given a measurable set $S\subset \mathbb{R}^n$, $|S|$ denotes the Lebesgue measure and $\chi_S$
means the characteristic function.
We also use the notations $j\wedge j'=\min\{j,j'\}$ and $j\vee j'=\max\{j,j'\}$.
Fix an integer $d\geq d_{p(\cdot)}\equiv \min\{d\in \mathbb{N}\bigcup\{0\}: p^-(n+d+1)>n\}.$
A function $a$ on $\mathbb{R}^n$ is called a $(p(\cdot),q)$-atom, if there exists a cube $Q$
such that
${\rm supp}\,a\subset Q$;
$\|a\|_{L^q}\leq \frac{|Q|^{1/q}}{\|\chi_{Q}\|_{L^{p(\cdot)}(\mathbb{R}^n)}}$;
$\int_{\mathbb{R}^n} a(x)x^\alpha dx=0\; {\rm for}\;|\alpha| \leq d$.
We say that a cube $Q\subset\mathbb R^n$ is dyadic if
$Q=Q_{j{\bf k}}=\{x=(x_1,x_2,\ldots,x_n)\in \mathbb R^n
:2^{-j-N}k_i\le x_i< 2^{-j-N}(k_i+1),i=1,2,\ldots,n\}$ for
some $j\in\mathbb Z$, some fixed positive large integer $N$ and ${\bf k}=(k_1,k_2,\ldots,k_n)\in\mathbb Z^n$.
Denote by $\ell(Q)=2^{-j}$ the side length of $Q=Q_{j{\bf k}}$.
Denote by $z_Q=2^{-j}{\bf k}$
the left lower corner of $Q$
and by $x_Q$ is any point in $Q$ when $Q=Q_{j{\bf k}}$. For any function
$\phi$ defined on $\mathbb R^n,$ $j\in\mathbb Z$, and $Q=Q_{j{\bf k}}$, set
\begin{align*}
&\psi_j(x)=2^{jn}\psi(2^{j}x),\\
&\tilde{\psi}(x)=\overline{\psi(-x)}\\
&\psi_Q(x)=|Q|^{1/2}\psi_j({x-z_Q}).
\end{align*}

The remainder of this paper is organized
as follows. In Section 2 we introduce the precise
definition of the Carleson measure space $CMO^{p(\cdot)}$
and establish the Plancherel-P\^{o}lya inequality for such space.
In Section 3, we introduce sequence spaces with variable exponents
$s^{p(\cdot)}$ and $c^{p(\cdot)}$ and obtain the duality
of the variable Hardy space $H^{p(\cdot)}$ with $CMO^{p(\cdot)}$ by a constructive proof, which is the heart of the present paper.
We show that Carleson measure spaces with variable exponents $CMO^{p(\cdot)}$, Campanato space with variable exponent
$\mathfrak{L}_{q,p(\cdot),d}$ and H\"older-Zygmund spaces with variable exponents $\mathcal H_d^{p(\cdot)}$ coincide as sets and the corresponding norms are equivalent in Section 4.
In Section 5, we discuss the boundedness of
Calder\'{o}n-Zygmund singular integral operators
on $CMO^{p(\cdot)}$ via
using an argument of weak
density property.

\section{Carleson measure space with variable exponent}

In this section, we introduce the Carleson measure space $CMO^{p(\cdot)}$.
To set notation,
let $\varphi,\psi\in\mathcal S$
satisfy
\begin{align}\label{s2d1}
\begin{split}
\mbox{supp}(\hat\varphi,\hat\psi)
\subset\{\xi\in\mathbb R^n:1/2\le|\xi|\le2\},\\
|\hat\varphi,\hat\psi(\xi)|\ge C>0\quad \mbox{if}\; \frac{3}{5}\le
\xi\le \frac{5}{3}\\
\noindent\mbox{and}\quad\quad\qquad
\sum_{j\in\mathbb Z}\overline{\hat\varphi(2^j\xi)}{\hat\psi(2^j\xi)}=1
\quad \mbox{if}\; \xi\neq0.
\end{split}
\end{align}

First we recall the well-known discrete Calder\'on identity
introduced by Frazier and
Jawerth \cite{FJ1}.

\begin{lemma}\label{s2l1}\quad Let $\varphi,\psi\in \mathcal S$
satisfy (\ref{s2d1}). Then
\begin{align*}
f(x)=\sum_{j\in \mathbb Z}2^{-jn}\sum_{\mathbf k\in\mathbb Z^n}\tilde\varphi_j\ast f(2^{-j}\mathbf{k})\psi_j(x-2^{-j}\mathbf{k})
=\sum_{Q}\left<f,\varphi_Q\right>\psi_Q(x),
\end{align*}
where
the series converges in $\mathcal S_\infty$
for all $f\in \mathcal {S_\infty}$.
Furthermore, the convergence of the right-hand, as well as the equality, is in $\mathcal{S}_\infty'$.
\end{lemma}

We would like to point out that functions $\varphi,\psi\in \mathcal S$
used in
the discrete Calder\'on identity do not have compact support. To prove the main result, we will also need the following new discrete Calder\'on-type identity, which, for the setting of spaces of homogeneous type, was first used in \cite{DH}.

\begin{lemma}\cite{T17AFA}\label{s2le2}\quad Suppose that $p(\cdot)\in LH$. Let $\phi$ be Schwartz functions with support
on the unit ball satisfying the conditions: for all $\xi\in\mathbb R^n$,
\begin{align*}
\sum_{j\in\mathbb Z}|\widehat\phi(2^{-j}\xi)|^2=1
\end{align*}
and $\int_{\mathbb R^n}\phi(x)x^\alpha dx=0$ for
all $0\le|\alpha|\le M$.
Then for all $f\in H^{p(\cdot)}\cap L^q$, $1<q<\infty$, there exists a function
$h\in H^{p(\cdot)}\cap L^{q}$ with
\begin{align*}
\|f\|_{L^{q}}\sim \|h\|_{L^{q}}
\quad\mbox{and}\quad\|f\|_{H^{p(\cdot)}}\sim \|h\|_{H^{p(\cdot)}}
\end{align*}
such that for some large integer $N$ depending on $\phi$ and $p(\cdot),q,M$,
\begin{align*}
f(x)=\sum_{j\in \mathbb Z}\sum_{Q}|Q|\phi_j\ast
h(x_Q)\phi_j(x-x_Q),
\end{align*}
where the series converges in both norms of $L^q$
and $H^{p(\cdot)}$.
\end{lemma}

We also recall the key estimate for norms of characteristic functions on variable
Lebesgue spaces.
\begin{lemma}[\cite{I}]\label{s2l2}Let $p(\cdot)\in \mathcal{B}$,
then there exist constant $C>0$ such that for all balls $B$ in
$\mathbb{R}^n$ and all measurable subsets $S\subset B$,\\
$$\displaystyle\frac{\|\chi_B\|_{L^{p(\cdot)}}}
{\|\chi_S\|_{L^{p(\cdot)}}}\leq C\frac{|B|}{|S|}$$.
\end{lemma}

We also need the following generalized H\"{o}lder inequality on variable Lebesgue spaces.

\begin{lemma}\label{s2le5} \cite{CF,TLZ}\quad Given exponent function $p_i(\cdot)\in \mathcal{P}^0,$ define
$p(\cdot)\in \mathcal{P}^0$ by
$$\frac{1}{p(x)}=\sum_{i=1}^m\frac{1}{p_i(x)},$$
where $i=1,2,\ldots,m.$
Then for all $f_i\in L^{p_i(\cdot)}$ and
$f_i\in L^{p(\cdot)}$ and
$$\|\prod_{i=1}^mf_i\|_{p(\cdot)}\leq C\prod_{i=1}^m\|f_i\|_{p_i(\cdot)}.$$
\end{lemma}

Now we recall the following boundedness of the vector-valued maximal operator $M$.
\begin{lemma}\cite{CFMP}\label{s2p2}\quad
Let $p(\cdot)\in LH\cap\mathcal P^0$.
Then
for any $q>1$, $f=\{f_i\}_{i\in \mathbb{Z}}$, $f_i\in L_{loc}$, $i\in \mathbb{Z}$
\begin{equation*}
  \|\|\mathbb{M}(f)\|_{l^q}\|_{L^{p(\cdot)}}\leq C\|\|f||_{l^q}\|_{L^{p(\cdot)}},
\end{equation*}
where $\mathbb{M}(f)=\{M(f_i)\}_{i\in\mathbb{Z}}$.
\end{lemma}

We now introduce a new space $CMO^{p(\cdot)}$ as follows.

\begin{definition}
Let $\psi\in \mathcal S$
satisfy (\ref{s2d1}), and $0<p^-\le p^+<\infty$. The Carleson measure space
$CMO^{p(\cdot)}$ is the collection of all $f\in \mathcal S'_\infty$ fulfilling
$$
\|f\|_{CMO^{p(\cdot)}}:=\sup_{P}\left\{\frac{|P|}{\|\chi_P\|^2_{p(\cdot)}}
\int_{\mathbb R^n}\sum_{Q\subset
P}|Q|^{-1}|\left<f, \psi_Q\right>|^2\chi_{Q}(x)dx\right\}^{1/2}<\infty.
$$
\end{definition}

The definition of $CMO^{p(\cdot)}$ is independent of the choice of $\{\psi_j\}_{j\in\mathbb Z}$
due to the following Plancherel-P\^{o}lya inequality for $CMO^{p(\cdot)}$.

\begin{theorem}\label{s2t1}
Let $\{\phi_j\}_j$ and $\{\varphi_k\}_k$ be any kernel functions satisfying
(\ref{s2d1}), and $p(\cdot)\in LH$, $0<p^-\le p^+<\infty$. Then for all $f\in\mathcal S'_\infty$,
\begin{align*}
\sup_{P}&\left\{\frac{|P|}{\|\chi_P\|^2_{p(\cdot)}}
\sum_{j=-\log_2\ell(P)}^\infty\sum_{Q\subset P}\left(\sup_{z\in Q}
|\tilde\phi_j\ast f(z)|\right)^2|Q|\right\}^{1/2}\\
\sim&
\sup_{P}\left\{\frac{|P|}{\|\chi_P\|^2_{p(\cdot)}}
\sum_{j=-\log_2\ell(P)}^\infty\sum_{Q\subset P}\left(\inf_{z\in Q}
|\tilde\varphi_j\ast f(z)|\right)^2|Q|\right\}^{1/2}.
\end{align*}
\end{theorem}

\begin{proof}
For any $f\in \mathcal {S'_\infty}$, we recall a
wavelet Calder\'on reproducing formula developed by Deng and Han \cite{DH},
\begin{align*}
f(x)=\sum_{j\in \mathbb Z}2^{-jn}\sum_{Q}\tilde\varphi_j\ast f(x_Q)
\psi_j(x,x_Q)
\end{align*}
where the series converges in $L^2$, $\mathcal S_\infty$ and $\mathcal S'_\infty$.
Then we rewrite $\tilde\phi_j\ast f(z)$ as
\begin{align*}
&\tilde\phi_j\ast f(z)\\
&=\sum_{Q'}\left<f,\varphi_{Q'}\right>\tilde\phi_j\ast\psi_{Q'}(z)\\
&=\sum_{j'\in\mathbb Z}\sum_{Q'}|Q'|^{1/2}\left<f,|Q'|^{1/2}\varphi_{j'}{(x-x_{Q'})}\right>
\int_{\mathbb R^n}\tilde\phi_j(z-x)\psi_{j'}(x,x_{Q'})dx.
\end{align*}

Here and below, we will apply the
almost orthogonal estimate which can be found in many monographs. For example, please see
\cite{g} for more details. To be more precise,
for any given positive integers $L$
and $\psi,\varphi\in\mathcal S$ satisfying cancellation conditions,
then
\begin{equation}\label{s2i2}
|\psi_j\ast \varphi_{j'}(x)|\le C\frac{2^{-|j-j'|L}2^{(j\wedge j')n}}
{(1+2^{(j\wedge j')}|x|)^{(n+M)}}.
\end{equation}

Using the inequality $(\ref{s2i2})$,
for any given positive integers $L,\;M$, we obtain
\begin{equation}\label{s2i3}
\int_{\mathbb R^n}\tilde\phi_j(z-x)\psi_{j'}(x,x_{Q'})dx
\le C\frac{2^{-|j-j'|L}2^{-(j\wedge j')M}}
{(2^{-(j\wedge j')}+|z-x_{Q'}|)^{(n+M)}}.
\end{equation}

Therefore,
\begin{align*}
|\tilde\phi_j\ast f(z)|
&\le C\sum_{j'\in\mathbb Z}
\sum_{Q'}
|Q'|\left<f,\varphi_{j'}{(x-x_{Q'})}\right>
\frac{2^{-|j-j'|L}2^{-(j\wedge j')M}}
{(2^{-(j\wedge j')}+|z-x_{Q'}|)^{(n+M)}}.
\end{align*}

Through the proof, $Q$ and $Q'$ always denote the dyadic cubes
with side length $2^{-j}$ and $2^{-j'}$, respectively. Hence,
for $x\in Q$,
\begin{align*}
|\tilde\phi_j\ast f(z)|
&\le C\sum_{j'\in\mathbb Z}
\sum_{Q'}
|Q'||\tilde\varphi_{j'}\ast f{(x_{Q'})}|
\frac{2^{-|j-j'|L}2^{-(j\wedge j')M}}
{(2^{-(j\wedge j')}+|x_Q-x_{Q'}|)^{(n+M)}}.
\end{align*}

Thus, applying H\"older's inequality yields
\begin{align*}
&(\sup_{z\in Q}|\tilde\phi_j\ast f(z)|)^2\\
\le& C\Bigg(\sum_{j'\in\mathbb Z}2^{-|j-j'|L}
\Bigg\{\sum_{Q'}
|Q'|\frac{2^{-(j\wedge j')}}
{(2^{-(j\wedge j')}+|x_Q-x_{Q'}|)^{(n+1)}}\Bigg\}^{1/2}\\
&\times
\Bigg\{\sum_{Q'}
|Q'||\tilde\varphi_{j'}\ast f{(x_{Q'})}|^2
\frac{2^{-(j\wedge j')M}}
{(2^{-(j\wedge j')}+|x_Q-x_{Q'}|)^{(n+M)}}\Bigg\}^{1/2}\Bigg)^2.
\end{align*}

Observe that
\begin{align}\label{s2i4}
\begin{split}
&\sum_{Q'}
|Q'|\frac{2^{-(j\wedge j')}}
{(2^{-(j\wedge j')}+|x_Q-x_{Q'}|)^{(n+1)}}\\
&\leq C\sum_{Q'}
\int_{\mathbb R^n}\frac{2^{-(j\wedge j')}}
{(2^{-(j\wedge j')}+|x_Q-y|)^{(n+1)}}
\chi_{Q'}(y)dy\\
&\leq C\int_{\mathbb R^n}\frac{2^{-(j\wedge j')}}
{(2^{-(j\wedge j')}+|x_Q-y|)^{(n+1)}}dy\le C.
\end{split}
\end{align}

Since $x_{Q'}$ can be replaced by any point in $Q$ in the
discrete Calder\'on identity,
by H\"older's inequality we get that
\begin{align*}
&(\sup_{z\in Q}|\tilde\phi_j\ast f(z)|)^2\\
\le& C\sum_{j'\in\mathbb Z}
\sum_{Q'}
2^{-|j-j'|L}
|Q'|\frac{2^{-(j\wedge j')M}}
{(2^{-(j\wedge j')}+|x_Q-x_{Q'}|)^{(n+M)}}
(\inf_{z\in Q'}|\tilde\varphi_{j'}\ast f{(z)}|)^2.
\end{align*}

Given a dyadic cube $P$ with $\ell(P)=2^{-k_0-N}$, we obtain
\begin{align*}
&\frac{|P|}{\|\chi_P\|^2_{p(\cdot)}}
\sum_{j=-\log_2\ell(P)}^\infty\sum_{\substack{Q\subset P\\
\ell(Q)=2^{-j-N}}}\left(\sup_{z\in Q}
|\tilde\phi_j\ast f(z)|\right)^2|Q|\\
\le& C\frac{|P|}{\|\chi_P\|^2_{p(\cdot)}}\sum_{j=k_0}^\infty
\sum_{\substack{Q\subset P\\
\ell(Q)=2^{-j-N}}}\sum_{j'=k_0}^\infty
\sum_{\substack{Q'\\\ell{(Q')}=2^{-j'-N}}}
2^{-|j-j'|L}
|Q'|\\
&\times\frac{2^{-(j\wedge j')M}}
{(2^{-(j\wedge j')}+|x_Q-x_{Q'}|)^{(n+M)}}
(\inf_{z\in Q'}|\tilde\varphi_{j'}\ast f{(z)}|)^2|Q|\\
&+C\frac{|P|}{\|\chi_P\|^2_{p(\cdot)}}\sum_{j=k_0}^\infty
\sum_{\substack{Q\subset P\\
\ell(Q)=2^{-j-N}}}\sum_{j'=-\infty}^{k_0-1}
\sum_{\substack{Q'\\\ell{(Q')}=2^{-j'-N}}}
2^{-|j-j'|L}
|Q'|\\
&\times\frac{2^{-(j\wedge j')M}}
{(2^{-(j\wedge j')}+|x_Q-x_{Q'}|)^{(n+M)}}
(\inf_{z\in Q'}|\tilde\varphi_{j'}\ast f{(z)}|)^2|Q|\\
=:&I+II.
\end{align*}

To prove the desired result, it suffices to set $M=1$ in the term $I$.
Furthermore, $I$ can be decomposed as
\begin{align*}
I=&C\frac{|P|}{\|\chi_P\|^2_{p(\cdot)}}\sum_{j=k_0}^\infty
\sum_{\substack{Q\subset P\\
\ell(Q)=2^{-j-N}}}\sum_{j'=k_0}^\infty
\sum_{\substack{Q'\subset3P\\\ell{(Q')}=2^{-j'-N}}}
2^{-|j-j'|L}|Q'|\\
&\times\frac{2^{-(j\wedge j')}}
{(2^{-(j\wedge j')}+|x_Q-x_{Q'}|)^{(n+1)}}
(\inf_{z\in Q'}|\tilde\varphi_{j'}\ast f{(z)}|)^2|Q|\\
&+C\frac{|P|}{\|\chi_P\|^2_{p(\cdot)}}\sum_{j=k_0}^\infty
\sum_{\substack{Q\subset P\\
\ell(Q)=2^{-j-N}}}\sum_{j'=k_0}^{\infty}
\sum_{\substack{Q'\cap 3P\neq\varnothing\\\ell{(Q')}=2^{-j'-N}}}
2^{-|j-j'|L}|Q'|\\
&\times\frac{2^{-(j\wedge j')}}
{(2^{-(j\wedge j')}+|x_Q-x_{Q'}|)^{(n+1)}}
(\inf_{z\in Q'}|\tilde\varphi_{j'}\ast f{(z)}|)^2|Q|\\
&=:I_1+I_2.
\end{align*}

Observe that
\begin{align*}
\sum_{\substack{Q'\subset 3P\\
\ell(Q')=2^{-j-N}}}
(\inf_{z\in Q'}|\tilde\varphi_{j'}\ast f{(z)}|)^2|Q'|
=&\sum_{\substack{Q'\subset 3P\\
\ell(Q')\le\ell{(P)}}}
(\inf_{z\in Q'}|\tilde\varphi_{j'}\ast f{(z)}|)^2|Q'|\\
\leq&C\sup_{\substack{P'\subset 3P\\
\ell(P')=\ell{(P)}}}\sum_{\substack {Q'\subset P'\\
\ell(Q')\le\ell{(P')}}}
(\inf_{z\in Q'}|\tilde\varphi_{j'}\ast f{(z)}|)^2|Q'|\\
\end{align*}

Thus,
\begin{align*}
I_1\le&C\frac{|P|}{\|\chi_P\|^2_{p(\cdot)}}\sum_{j=k_0}^\infty
\sum_{\substack{Q\subset P\\
\ell(Q)=2^{-j-N}}}\sum_{j'=k_0}^\infty
\sup_{\substack{P'\subset 3P\\
\ell(P')=\ell{(P)}}}\sum_{\substack {Q'\subset P'\\
\ell(Q')\le\ell{(P')}}}
2^{-|j-j'|L}|Q'|\\
&\times\frac{2^{-(j\wedge j')}}
{(2^{-(j\wedge j')}+|x_Q-x_{Q'}|)^{(n+1)}}
(\inf_{z\in Q'}|\tilde\varphi_{j'}\ast f{(z)}|)^2|Q|\\
\le&C\frac{|P|}{\|\chi_P\|^2_{p(\cdot)}}\sum_{j=k_0}^\infty
\sum_{j'=k_0}^\infty
\sup_{\substack{P'\subset 3P\\
\ell(P')=\ell{(P)}}}\sum_{\substack {Q'\subset P'\\
\ell(Q')\le\ell{(P')}}}
2^{-|j-j'|L}|Q'|
(\inf_{z\in Q'}|\tilde\varphi_{j'}\ast f{(z)}|)^2\\
\le&C\sup_{P'}
\frac{|P'|}{\|\chi_P'\|^2_{p(\cdot)}}
\sum_{j'=-\log_2\ell{(P')}}^\infty
\sum_{\substack {Q'\subset P'\\
\ell(Q')=2^{-j'-N}}}|Q'|
(\inf_{z\in Q'}|\tilde\varphi_{j'}\ast f{(z)}|)^2
\end{align*}
where the first inequality follows from
the estimate $(\ref{s2i4})$.

Next we decompose the set of dyadic cubes
$\{R:R\cap3P=\varnothing,\ell(R)=\ell(P)\}$
into $\{B_i\}_{i\in\mathbb N}$. Namely, for each $i\in\mathbb N$,
$$
B_i:=\{P':P'\cap3P=\varnothing,\ell(P)=\ell(P'),
2^{i-k_0-N}\le|y_{P'}-y_P|\le2^{i-k_0-N+1}\},
$$
where $y_{P'}$ and $y_P$ denote the center of $P'$
and $P$, respectively.
Then, we obtain
\begin{align*}
I_2=&C\frac{|P|}{\|\chi_P\|^2_{p(\cdot)}}\sum_{j=k_0}^\infty
\sum_{\substack{Q\subset P\\
\ell(Q)=2^{-j-N}}}\sum_{j'=k_0}^{\infty}
\sum_{i=1}^\infty\sum_{P'\in B_i}
\sum_{\substack{Q'\subset P'\\\ell{(Q')}=2^{-j'-N}}}
2^{-|j-j'|L}|Q'|\\
&\times\frac{2^{-(j\wedge j')}}
{(2^{-(j\wedge j')}+|x_Q-x_{Q'}|)^{(n+1)}}
(\inf_{z\in Q'}|\tilde\varphi_{j'}\ast f{(z)}|)^2|Q|\\
\leq&C\sum_{i=1}^\infty\sum_{P'\in B_i}\frac{|P'|}{\|\chi_P'\|^2_{p(\cdot)}}\sum_{j=k_0}^\infty
\sum_{\substack{Q\subset P\\
\ell(Q)=2^{-j-N}}}\sum_{j'=k_0}^{\infty}
\sum_{\substack{Q'\subset P'\\\ell{(Q')}=2^{-j'-N}}}
2^{-|j-j'|L}|Q'|\\
&\times\frac{2^{-(j\wedge j')}}
{(2^{-(j\wedge j')}+|x_P-x_{P'}|)^{(n+1)}}
(\inf_{z\in Q'}|\tilde\varphi_{j'}\ast f{(z)}|)^2|Q|\\
\end{align*}

Observe that $\sum_{\substack{Q\subset P\\\ell{(Q)}=2^{-j}}}
|Q|=|P|$ for each $j\ge k_0$
and there are at most $2^{(i+1)n}$ cubes in $B_i$. Hence,
\begin{align*}
I_2
\leq&C\sum_{i=1}^\infty\sum_{P'\in B_i}|P|
\frac{2^{-k_0}}{2^{(i-k_0)(n+1)}}
\sum_{j'=k_0}^\infty
\bigg(\sum_{j=k_0}^\infty 2^{k_0-(j\wedge j')-L|j-j'|}\bigg)\\
&\times\bigg(\frac{|P'|}{\|\chi_P'\|^2_{p(\cdot)}}
\sum_{\substack{Q'\subset P'\\\ell{(Q')}=2^{-j'-N}}}
(\inf_{z\in Q'}|\tilde\varphi_{j'}\ast f{(z)}|)^2|Q'|\bigg)\\
\leq&C \sup_{P'}\bigg(\frac{|P'|}{\|\chi_P'\|^2_{p(\cdot)}}
\sum_{j'=k_0}^\infty
\sum_{\substack{Q'\subset P'\\\ell{(Q')}=2^{-j'-N}}}
(\inf_{z\in Q'}|\tilde\varphi_{j'}\ast f{(z)}|)^2|Q'|\bigg)\\
&\times\bigg(\sum_{i=1}^\infty|P|2^{in}
\frac{2^{-k_0}}{2^{(i-k_0)(n+1)}}\bigg)\\
\leq&C \sup_{P'}\frac{|P'|}{\|\chi_P'\|^2_{p(\cdot)}}
\sum_{j'=-\log_2\ell(P')}^\infty
\sum_{\substack{Q'\subset P'\\\ell{(Q')}=2^{-j'-N}}}
(\inf_{z\in Q'}|\tilde\varphi_{j'}\ast f{(z)}|)^2|Q'|.
\end{align*}

Now we deal with $II$.
Since for every integer $j'<k_0$, we denote $$j'\equiv k_0-m,
\quad\mbox{for}\quad m\in \mathbb N.$$
Let $$E_{m}^0=\{Q':\ell(Q')=2^{m-k_0-N};\;|y_P-y_{Q'}|\le2^{m-k_0-N-1}\}$$
and $$E_{m}^i=\{Q':\ell(Q')=2^{m-k_0-N};\;2^{i+m-k_0-N-1}<|y_P-y_{Q'}|
\le2^{i+m-k_0-N}\}$$
for $i\in \mathbb N$.
We rewrite

\begin{align*}
II\leq&C\frac{|P|}{\|\chi_P\|^2_{p(\cdot)}}\sum_{j=k_0}^\infty
|P|\sum_{j'=-\infty}^{k_0-1}
\sum_{\substack{Q'\\\ell{(Q')}=2^{-j'-N}}}
2^{-|j-j'|L}
|Q'|\\
&\times\frac{2^{-(j\wedge j')M}}
{(2^{-(j\wedge j')}+|y_P-y_{Q'}|)^{(n+M)}}
(\inf_{z\in Q'}|\tilde\varphi_{j'}\ast f{(z)}|)^2\\
=&C\frac{|P|}{\|\chi_P\|^2_{p(\cdot)}}\sum_{j=k_0}^\infty
2^{-k_0n}\sum_{m=1}^{\infty}\sum_{i=0}^\infty
\sum_{E_{m}^i}
2^{-(j-k_0+m)L}
|Q'|\\
&\times\frac{2^{(m-k_0)M}}
{(2^{m-k_0}+|y_P-y_{Q'}|)^{(n+M)}}
(\inf_{z\in Q'}|\tilde\varphi_{k_0-m}\ast f{(z)}|)^2.
\end{align*}

There are at most $2^{in}+3^n$ dyadic cubes $Q'\in E_{m,j'}^i$
for $i=\mathbb N\cup \{0\}$.
We can choose $P'=10002^{i+m}P$ such that $P'\supseteq Q'$
and $P'\supseteq P$.
Therefore,
\begin{align*}
II\leq&C\frac{|P|}{\|\chi_P\|^2_{p(\cdot)}}\sum_{j=k_0}^\infty
2^{-k_0n}\sum_{m=1}^{\infty}\sum_{i=0}^\infty
\sum_{E_{m}^i}
2^{-(j-k_0+m)L}\\
&\times
\frac{2^{(m-k_0)M}}
{2^{(i+m-k_0)(n+M)}}
(\inf_{z\in Q'}|\tilde\varphi_{k_0-m}\ast f{(z)}|)^2|Q'|\\
\leq&C\frac{|P|}{\|\chi_P\|^2_{p(\cdot)}}\sum_{j=k_0}^\infty
2^{-k_0n}\sum_{m=1}^{\infty}\sum_{i=0}^\infty
(2^{in}+3^n)
2^{-(j-k_0+m)L}\\
&\times
\frac{2^{(m-k_0)M}}{2^{(i+m-k_0)(n+M)}}
\frac{\|\chi_{P'}\|^2_{p(\cdot)}}{|P'|}
\frac{|P'|}{\|\chi_{P'}\|^2_{p(\cdot)}}
(\inf_{z\in Q'}|\tilde\varphi_{k_0-m}\ast f{(z)}|)^2|Q'|\\
\leq&C\sum_{j=k_0}^\infty
2^{-k_0n}\sum_{m=1}^{\infty}\sum_{i=0}^\infty
\frac{\|\chi_{P'}\|^2_{p(\cdot)}}{\|\chi_{P}\|^2_{p(\cdot)}}
\frac{|P|}{|P'|}
(2^{in}+3^n)
2^{-(j-k_0+m)L}\frac{2^{(m-k_0)M}}
{2^{(i+m-k_0)(n+M)}}\\
&\times
\bigg(\sup_{P'}\frac{|P'|}{\|\chi_{P'}\|^2_{p(\cdot)}}
\sum_{j'=-\log_2\ell(P')}^\infty
\sum_{\substack{Q'\subset P'\\\ell{(Q')}=2^{-j'-N}}}
(\inf_{z\in Q'}|\tilde\varphi_{j'}\ast f{(z)}|)^2|Q'|\bigg).
\end{align*}

By Lemma \ref{s2l2}, we have
\begin{align*}
&\frac{\|\chi_{P'}\|^2_{p(\cdot)}}{\|\chi_{P}\|^2_{p(\cdot)}}
\frac{|P|}{|P'|}
=\frac{\|\chi_{P'}\|^{2/{p^-}}_{p(\cdot)/{p^-}}}
{\|\chi_{P}\|^{2/{p^-}}_{p(\cdot)/{p^-}}}
\frac{|P|}{|P'|}\\
\le& \frac{|P'|}{|P|}^{2/{p^-}-1}
=C2^{(2/{p^-}-1)(i+m)n}.
\end{align*}

Set $L>\max\{1,n(2/{p^-}-2)\}$ and $M>n(2/{p^-}-1)$.
Observe that
$$\sum_{i=0}^\infty2^{in(2/{p^-}-1)}(2^{in}+3^n)2^{-i(n+M)}\le C$$
and
$$\sum_{j=k_0}^\infty2^{-jL}\le C2^{-k_0L};
\quad
\sum_{m=1}^{\infty}2^{mn(2/{p^-}-1)}
2^{-mL}{2^{-mn}}\le C.$$
Then we get
\begin{align*}
&\sum_{j=k_0}^\infty
2^{-k_0n}\sum_{m=1}^{\infty}\sum_{i=0}^\infty
\frac{\|\chi_{P'}\|^2_{p(\cdot)}}{\|\chi_{P}\|^2_{p(\cdot)}}
\frac{|P|}{|P'|}
(2^{in}+3^n)
2^{-(j-k_0+m)L}\frac{2^{(m-k_0)M}}
{2^{(i+m-k_0)(n+M)}}\\
&\le\sum_{j=k_0}^\infty
2^{-k_0n}\sum_{m=1}^{\infty}\sum_{i=0}^\infty
2^{(2/{p^-}-1)(i+m)n}
(2^{in}+3^n)
2^{-(j-k_0+m)L}\frac{2^{(m-k_0)M}}
{2^{(i+m-k_0)(n+M)}}\\
&\le
2^{-k_0(n+L)}\sum_{m=1}^{\infty}2^{mn(2/{p^-}-1)}
2^{-(-k_0+m)L}\frac{2^{(m-k_0)M}}
{2^{(m-k_0)(n+M)}}\le C.
\end{align*}

The proof of the Plancherel-P\^{o}lya inequality for $CMO_{p(\cdot)}$ is
complete.
\end{proof}

By Theorem \ref{s2t1}, we immediately obtain the following
discrete version of $CMO^{p(\cdot)}$.

\begin{corollary}\label{s2c1}
Let $\{\varphi_j\}_j$ be any kernel functions satisfying
(\ref{s2d1}), and $p(\cdot)\in LH$, $0<p^-\le p^+<\infty$. Then for all $f\in CMO^{p(\cdot)}$,
\begin{align*}
\|f\|_{CMO^{p(\cdot)}}
\sim&
\sup_{P}\left\{\frac{|P|}{\|\chi_P\|^2_{p(\cdot)}}
\sum_{j=-\log_2\ell(P)}^\infty\sum_{Q\subset P}
|\varphi_j\ast f(x_Q)|^2|Q|\right\}^{1/2},
\end{align*}
where $x_Q$ is any fixed point in $Q$.
\end{corollary}

\section{Duality of $H^{p(\cdot)}$ and $CMO^{p(\cdot)}$}

Define a linear map $S_\varphi$ by
$$
S_\varphi (f)=\{\left<f,\varphi_Q\right>\}_Q,
$$
and another linear map $T_{\psi}$ by
$$
T_\psi(\{s_Q\}_Q)=\sum_{Q}s_Q\psi_Q.
$$
For $g\in CMO^{p(\cdot)}$, define a linear functional
$L_g$ by
$$
L_g(f)=\left<S_\psi(g),S_\varphi(f)\right>=\sum_{Q}
\left<g,\psi_Q\right>\left<f,\varphi_Q\right>
$$
for $f\in\mathcal S_\infty$.

We now state the following main result in this section.
\begin{theorem}\label{s3t1}
Suppose that $p(\cdot)\in LH$, $0<p^-\le p^+\leq1$. The dual of $H^{p(\cdot)}$
is $CMO^{p(\cdot)}$ in the following sense.\\
\noindent (1) For $g\in CMO^{p(\cdot)}$, the linear functional $L_g$,
defined initially on $\mathcal S_\infty$, extends to a continuous linear functional
on $H^{p(\cdot)}$ with $\|L_g\|\le C\|g\|_{CMO^{p(\cdot)}}$.

\noindent (2) Conversely, every continuous linear functional $L$ on $H^{p(\cdot)}$
satisfies $L=L_g$ for some $g\in CMO^{p(\cdot)}$ with $\|g\|_{CMO^{p(\cdot)}}
\le C\|L\|$.
\end{theorem}

To prove this theorem,
we first introduce sequence spaces with variable exponents.
For $p(\cdot)\in LH$, $0<p^-\le p^+\leq1$, the sequence space
$s^{p(\cdot)}$ consists all complex-value sequences
$$
s^{p(\cdot)}=\left\{
\{s_Q\}_Q:\|s_Q\|_{s^{p(\cdot)}}:=
\left\|\bigg\{\sum_Q
|s_Q|^2|Q|^{-1}\chi_Q\bigg\}^{1/2}\right\|_{L^{p(\cdot)}}<\infty
\right\};
$$
the sequence space $c^{p(\cdot)}$ consists all complex-value sequences
$$
c^{p(\cdot)}=\left\{
\{t_Q\}_Q:\|t_Q\|_{c^{p(\cdot)}}:=
\sup_{P}\left\{
\frac{|P|}{\|\chi_P\|^2_{L^{p(\cdot)}}}
\sum_{Q\subset P}|t_Q|^2
\right\}^{1/2}
<\infty
\right\}.
$$

We mention that, the sequence spaces $s^p$ and $c^1$ were first introduced
by Frazier and Jawerth (\cite{FJ2}), $c^p$ was introduced by
Lee et al (\cite{LLL}), the sequence space $f_{p(\cdot),q(\cdot)}^{s(\cdot),\phi}$
corresponding to the space $F_{p(\cdot),q(\cdot)}^{s(\cdot),\phi}$ was introduced
by Yang et al. (\cite{YZY}).
We also remark that Zhuo, Yang and Liang (\cite{ZYL}) showed that
$p(\cdot)-$Carleson measure characterizations are
presented for the dual space of $H^{p(\cdot)}$.
However, the main results and the methods used in our paper
are quite different. In order to prove
our main result in this section,
we need the following lemma.

\begin{lemma}[\cite{NS}]\label{s2l6}Assume that $p^+\leq 1$. For sequences of scalars $\{\lambda_j\}_{j=1}^{\infty}$
and sequences of $(p(\cdot),q)-$atoms $\{a_j\}$,
we have $$\sum_{j=1}^{\infty}|\lambda_j|\leq \mathcal{A}(\{\lambda_j\}_{j=1}^{\infty},\{Q_j\}_{j=1}^{\infty}),$$
where
$$
  \mathcal{A}(\{\lambda_j\}_{j=1}^\infty,\{Q_j\}_{j=1}^\infty)
  =\left\|\left\{\sum_{j}\left(\frac{|\lambda_j|\chi_{Q_j}}{\|\chi_{Q_j}\|_{L^{p(\cdot)}}}\right)^{p^-}
\right\}^{\frac{1}{p^-}}\right\|_{L^{p(\cdot)}}.
$$
\end{lemma}

To prove Theorem \ref{s3t1}, we also need the following
two propositions.
\begin{proposition}\label{s3p2}
Suppose that $p(\cdot)\in LH$, $0<p^-\le p^+\leq1$
and $\varphi,\;\psi$ satisfy (\ref{s2d1}).
The linear operator $S_\varphi: H^{p(\cdot)}\mapsto s^{p(\cdot)}$
and $T_\psi: s^{p(\cdot)}\mapsto H^{p(\cdot)}$, respectively, are
bounded. Furthermore, $T_\psi\circ S_\varphi$ is the identity on
$H^{p(\cdot)}$.
\end{proposition}

\begin{proposition}\label{s3p3}
Suppose that $p(\cdot)\in LH$, $0<p^-\le p^+\leq1$
and $\varphi,\;\psi$ satisfy (\ref{s2d1}).
The linear operator $S_\varphi: CMO^{p(\cdot)}\mapsto c^{p(\cdot)}$
and $T_\psi: c^{p(\cdot)}\mapsto CMO^{p(\cdot)}$, respectively, are
bounded. Furthermore, $T_\psi\circ S_\varphi$ is the identity on
$CMO^{p(\cdot)}$.
\end{proposition}

Assume the above two propositions first,
then we return the proof of Theorem \ref{s3t1}.

\noindent\textit{Proof of Theorem \ref{s3t1}.}
By Lemma \ref{s2le2},
for all $f\in \mathcal {S_\infty}$
\begin{align*}
f(x)=\sum_{j\in \mathbb Z}2^{-jn}\sum_{\mathbf k\in\mathbb Z^n}
\phi_j\ast h(2^{-j}\mathbf{k})\phi_j(x-2^{-j}\mathbf{k})
=\sum_{Q}\left<h,\phi_Q\right>\phi_Q(x),
\end{align*}
where
the series also converges in $\mathcal S_\infty$
and $\phi\in \mathcal S_\infty$ is defined in Lemma \ref{s2le2}.
Let $g\in CMO^{p(\cdot)}$ and $f\in H^{p(\cdot)}$.
Define a linear functional $L_g$ on $\mathcal S_\infty$ by
\begin{align*}
&L_g(f)=\left<f,g\right>
=\sum_{Q}\left<h,\phi_Q\right>\left<\phi_Q,g\right>.
\end{align*}

Now we need the maximal square function defined by
\begin{align*}
\mathcal{G}_\phi^d(h)(x):=\left(\sum_{j\in \mathbb Z}
\sum_{\mathbf k\in\mathbb Z^n}\sup_{x_Q\in Q}|\phi_j\ast h(x_Q)|^2\chi_Q(x)\right)^{1/2}.
\end{align*}

Set
\begin{align*}
\Omega_i=\{x\in\mathbb R^n:\mathcal{G}_\phi^d(h)(x)>2^i\}.
\end{align*}
and
\begin{align*}
\widetilde\Omega_i=\{x\in\mathbb R^n:M{(\chi_{\Omega_i})}(x)>\frac{1}{10}\},
\end{align*}
where $M$ is the Hady-Littlewood maximal operator. Then $\Omega_i\subset\widetilde{\Omega_i}$.
By the $L^2$ boundedness of $M$, $|\widetilde{\Omega_i}|\le C|\Omega_i|.$
Denote $$B_i=\{Q:|Q\cap\Omega_i|> \frac{1}{2}|Q|, |Q\cap\Omega_{i+1}|\le \frac{1}{2}|Q|\}.$$
Following the discrete Calder\'on reproducing formula and denoting
$\tilde Q\in B_i$ are maximal dyadic cubes in $B_i$, we rewrite
\begin{align*}
f(x)=\sum_{i\in \mathbb Z}\sum_{\tilde Q\in B_i}\sum_{Q\subset\tilde Q}
\phi_Q\ast
f(x_Q)\phi_Q(x-x_Q).
\end{align*}

From Corollary \ref{s2c1} and the H\"older inequality,
it follows that
\begin{align*}
|L_g(f)|
&=\bigg|\sum_{Q}\left<h,\phi_Q\right>\left<\phi_Q,g\right>\bigg|\\
&=\bigg|\sum_{j\in \mathbb Z}\sum_{Q}|Q|\phi_j\ast h(x_Q)
\phi_j\ast g(x_Q)\bigg|\\
&\le\sum_{i\in \mathbb Z}\sum_{\tilde Q\in B_i}\sum_{Q\subset\tilde Q}
|\phi_Q\ast h(x_Q)||\phi_Q\ast g(x)|\\
&\le\sum_{i\in \mathbb Z}\sum_{\tilde Q\in B_i}
\left\{\sum_{Q\subset\tilde Q}
|\phi_Q\ast h(x_Q)|^2\right\}^{\frac{1}{2}}
\left\{\sum_{Q\subset\tilde Q}|\phi_Q\ast g(x)|^2\right\}^{\frac{1}{2}}\\
&\le\sum_{i\in \mathbb Z}\sum_{\tilde Q\in B_i}
\left\{\sum_{Q\subset\tilde Q}
|\phi_Q\ast f(x_Q)|^2\right\}^{\frac{1}{2}}
\left\{
\frac{|\tilde Q|}{\|\chi_{\tilde Q}\|^2_{L^p(\cdot)}}
\sum_{Q\subset\tilde Q}|\phi_Q\ast g(x)|^2\right\}^{\frac{1}{2}}\\
&\le C
\sum_{i\in \mathbb Z}\sum_{\tilde Q\in B_i}
\frac{\|\chi_{\tilde Q}\|_{L^p(\cdot)}}{|\tilde Q|^{1/2}}
\left\{\sum_{Q\subset\tilde Q}
|\phi_Q\ast h(x_Q)|^2\right\}^{\frac{1}{2}}\|g\|_{CMO^{p(\cdot)}}.
\end{align*}

We denote that $$\lambda_{\tilde Q}:=
\frac{\|\chi_{\tilde Q}\|_{L^p(\cdot)}}{|\tilde Q|^{1/2}}
\left\{\sum_{Q\subset\tilde Q}
|\phi_Q\ast h(x_Q)|^2\right\}^{\frac{1}{2}}.$$

Then
\begin{align*}
f(x)=\sum_{i\in \mathbb Z}\sum_{\tilde Q\in B_i}\sum_{Q\subset\tilde Q}\phi_Q\ast h(x_Q)\phi_Q(x-x_Q)=:\sum_{i}\sum_{\tilde Q\in B_i}\lambda_{\tilde Q}a_{\tilde Q}(x),
\end{align*}
where $$a_{\tilde Q}=\frac{1}{\lambda_{\tilde Q}}
\sum_{Q\subset\tilde Q}\phi_Q\ast
h(x_Q)\phi_Q(x-x_Q).$$

Here we have established the atomic decomposition for $H^{p(\cdot)}$.
In fact, from the definition of $a_{\tilde Q}$ and the support of $\phi$, we get that
$a_{\tilde Q}$ is supported in $5\tilde Q$.

We claim that
\begin{equation*}
  \mathcal{A}(\{\lambda_j\}_{j=1}^\infty,\{Q_j\}_{j=1}^\infty)
  \leq C\|f\|_{{H}^{p(\cdot)}}.
\end{equation*}

To prove the claim, we first observe that when $1< q<\infty$
\begin{align*}
\mathcal{A}(\{\lambda_j\}_{j=1}^\infty,\{Q_j\}_{j=1}^\infty)
&=\left\|\left\{\sum_{i}\sum_{\tilde Q\in B_i}\left(
\frac{|\lambda_{\tilde Q}|\chi_{5\tilde Q}}{\|\chi_{5\tilde Q}\|_{L^{p(\cdot)}}}\right)^{p^-}
\right\}^{\frac{1}{p^-}}\right\|_{L^{p(\cdot)}}\\
&\le C\left\|\left\{\sum_{i}\sum_{\tilde Q\in B_i}\left(
\left(\sum_{Q\subset\tilde Q}
|\phi_Q\ast h(x_Q)|^2\right)^{\frac{1}{2}}\chi_{5\tilde Q}\right)^{p^-}
\right\}^{\frac{1}{p^-}}
\right\|_{L^{p(\cdot)}}.
\end{align*}

Note that
$5\tilde Q\subset \tilde \Omega_i$,
when $\tilde Q\in B_i$.
Since $\Omega_i\subset\tilde\Omega_i$ for each $i\in\mathbb Z$ and
$|\tilde\Omega_i|\le C|\Omega_i|$, for all $x\in\mathbb R^n$ we have
\begin{align*}
\chi_{\tilde\Omega_i}(x)\le CM^{\frac{2}{p^-}}\chi_{\Omega_{i}}(x).
\end{align*}

Applying Lemma \ref{s2p2}, we have that
\begin{align*}
&\mathcal{A}(\{\lambda_j\}_{j=1}^\infty,\{Q_j\}_{j=1}^\infty)\\
\le& C\left\|\left\{\sum_{i}\left(
\left(\sum_{Q\subset\tilde Q}
|\phi_Q\ast h(x_Q)|^2\right)^{\frac{1}{2}}
\chi_{\tilde \Omega_i}\right)^{p^-}
\right\}^{\frac{1}{p^-}}
\right\|_{L^{p(\cdot)}}\\
\le& C\left\|\left\{\sum_{i}\left(
\left(\sum_{Q\subset\tilde Q}
|\phi_Q\ast h(x_Q)|^2\right)^{\frac{1}{2}}M^{\frac{2}{p^-}}\chi_{\Omega_{i}}\right)^{p^-}
\right\}^{\frac{1}{p^-}}
\right\|_{L^{p(\cdot)}}\\
\le& C\left\|\left\{\sum_{i}
\left(\sum_{Q\subset\tilde Q}
|\phi_Q\ast h(x_Q)|^2\right)^{\frac{p^-}{2}}\chi_{\Omega_{i}}^2
\right\}^{\frac{1}{2}}
\right\|^{\frac{2}{p^-}}_{L^{\frac{2p(\cdot)}{p^-}}}\\
\le& C\left\|\left\{\sum_{i}\left(
\left(\sum_{Q\subset\tilde Q}
|\phi_Q\ast h(x_Q)|^2\right)^{\frac{1}{2}}\chi_{ \Omega_i}\right)^{p^-}
\right\}^{\frac{1}{p^-}}
\right\|_{L^{p(\cdot)}}.
\end{align*}

Observing that $\Omega_{i+1}\subset\Omega_i$ and
$|\bigcap_{i=1}^\infty\Omega_i|=0$,
then for $a.e~x\in\mathbb R^n$ we have
\begin{align*}
\sum_{i=-\infty}^{\infty}2^{i}\chi_{\Omega_i}(x)
&=\sum_{i=-\infty}^{\infty}2^i\sum_{j=i}^\infty\chi_{\Omega_j\backslash{\Omega_{j+1}}}(x)
=2\sum_{j=-\infty}^{\infty}2^j\chi_{\Omega_j\backslash{\Omega_{j+1}}}(x).
\end{align*}

By repeating the similar argument in the proof of \cite[Theorem 1.1]{T17AFA},
we have
\begin{align*}
&\left\|\left\{\sum_{i}\left(
\left(\sum_{Q\subset\tilde Q}
|\phi_Q\ast h(x_Q)|^2\right)^{\frac{1}{2}}
\chi_{ \Omega_i}\right)^{p^-}
\right\}^{\frac{1}{p^-}}
\right\|_{L^{p(\cdot)}}\\
&\le
C\left\|\left\{\sum_{i}\left(
\left(\sum_{Q\subset\tilde Q}
|\phi_Q\ast h(x_Q)|^2\right)^{\frac{1}{2}}\chi_{ \Omega_i\setminus\Omega_{i+1}}\right)^{p^-}
\right\}^{\frac{1}{p^-}}
\right\|_{L^{p(\cdot)}}\\
&\le C\inf\left\{\lambda>0:\int_{\mathbb R^n}\left(
\sum_i\frac{2^i\chi_{ \Omega_i\setminus\Omega_{i+1}}}{\lambda}
\right)^{p(x)}dx\le 1\right\}\\
&= C\inf\left\{\lambda>0:\sum_i\int_{{ \Omega_i\setminus\Omega_{i+1}}}
\left(\frac{2^i}{\lambda}
\right)^{p(x)}dx\le 1\right\}\\
&\le C\inf\left\{\lambda>0:\int_{\mathbb R^n}
\left(\frac{\mathcal G_\phi^df(x)}{\lambda}
\right)^{p(x)}dx\le 1\right\}\leq C\|f\|_{{H}^{p(\cdot)}}.
\end{align*}
Therefore, we have proved the claim. Moreover,
we can obtain that every $a_{\tilde Q}$ is a $(p(\cdot),q)-$atom.

Therefore,
\begin{align*}
|L_g(f)|
&=\bigg|\sum_{Q}\left<h,\phi_Q\right>\left<\phi_Q,g\right>\bigg|\\
&\le C
\sum_{i\in \mathbb Z}\sum_{\tilde Q\in B_i}
\frac{\|\chi_{\tilde Q}\|_{L^p(\cdot)}}{|\tilde Q|^{1/2}}
\left\{\sum_{Q\subset\tilde Q}
|\phi_Q\ast h(x_Q)|^2\right\}^{\frac{1}{2}}\|g\|_{CMO^{p(\cdot)}}\\
&\le C\sum_{j=1}^{\infty}|\lambda_j|\|g\|_{CMO^{p(\cdot)}}\\
&\leq \mathcal{A}(\{\lambda_j\}_{j=1}^{\infty},\{Q_j\}_{j=1}^{\infty})
\|g\|_{CMO^{p(\cdot)}}\\
&\le C\|f\|_{H^{p(\cdot)}}\|g\|_{CMO^{p(\cdot)}}.
\end{align*}
This shows that $g\in(H^{p(\cdot)})^\ast$
and
$$\|L_g\|\le C\|g\|_{CMO^{p(\cdot)}}.
$$

Conversely, we first prove that every continuous linear functional $\ell$ on $s^{p(\cdot)}$
satisfies $\ell=\ell_t$ for some $t\in c^{p(\cdot)}$ with $\|t\|_{c^{p(\cdot)}}
\le C\|\ell\|$.
For $s=\{s_Q\}_Q\in s^{p(\cdot)}$,
let $\ell(s)=\sum_{Q}s_Qt_Q$.
Fix a dyadic cube $P$ in $\mathbb R^n$.
Let $X$ be the sequence space consisting of
$s=\{s_Q\}_{Q\subset P}$, and define a counting measure on dyadic cubes
$Q\subset P$ by
$d\sigma(Q)=\frac{|Q|}{|P|^{-1}\|\chi_P\|^2_{L^{p(\cdot)}}}$.

Then
\begin{align*}
&\left(\frac{|P|}{\|\chi_P\|^2_{L^{p(\cdot)}}}
\sum_{Q\subset P}|t_Q|^2\right)^{1/2}\\
&=\left\||t_Q||Q|^{-1/2}\right\|_{l^2(X,d\sigma)}\\
&=\sup_{\|s\|_{l^2(X,d\sigma)}\le 1}
\left|\frac{|P|}{\|\chi_P\|^2_{L^{p(\cdot)}}}
\sum_{Q\subset P}|Q||s_Q||Q|^{-1/2}|t_Q|\right|\\
&=\sup_{\|s\|_{l^2(X,d\sigma)}\le 1}
\left|\ell\left(\frac{|P|}{\|\chi_P\|^2_{L^{p(\cdot)}}}
|Q||s_Q||Q|^{-1/2}\right)\right|\\
&=\|\ell\|\sup_{\|s\|_{l^2(X,d\sigma)}\le 1}
\left\|\left\{\frac{|P||s_Q||Q|^{1/2}}{\|\chi_P\|^2_{L^{p(\cdot)}}}
\right\}_{Q\subset P}\right\|_{s^{p(\cdot)}}.
\end{align*}

Choose that $0<r(x)<\infty$ such that
$\frac{1}{p(x)}=1+\frac{1}{r(x)}$.
By Lemma \ref{s2le5} and Lemma \ref{s2l2} and the H\"older inequality, we have

\begin{align*}
&\left\|\left\{\frac{|P||s_Q||Q|^{1/2}}{\|\chi_P\|^2_{L^{p(\cdot)}}}
\right\}_{Q\subset P}\right\|_{s^{p(\cdot)}}\\
&=
\left\|\left\{\sum_{Q\subset P}\frac{|P|^2|s_Q|^2}
{\|\chi_P\|^4_{L^{p(\cdot)}}}\chi_Q
\right\}^{1/2}\right\|_{L^{p(\cdot)}}\\
&\le C
\left\|\left\{\sum_{Q\subset P}\frac{|P|^2|s_Q|^2}
{\|\chi_P\|^4_{L^{p(\cdot)}}}\chi_Q
\right\}^{1/2}\right\|_{L^1}\|\chi_Q\|_{L^{r(\cdot)}}\\
&=C|P|\left\{\frac{1}{|P|}\int_P
\bigg(\sum_{Q\subset P}\frac{|s_Q|^2}
{(\|\chi_P\|^2_{L^{p(\cdot)}}|P|^{-1})^2}
\chi_Q(x)\bigg)^{1/2}dx\right\}\|\chi_Q\|_{L^{p(\cdot)}}|Q|^{-1}\\
&\le C|P|\left\{\frac{1}{|P|}\int_P
\bigg(\sum_{Q\subset P}\frac{|s_Q|^2}
{(\|\chi_P\|^2_{L^{p(\cdot)}}|P|^{-1})^2}
\chi_Q(x)\bigg)dx\right\}^{1/2}\|\chi_Q\|_{L^{p(\cdot)}}|Q|^{-1}
\end{align*}

Then,
\begin{align*}
&\left\|\left\{\frac{|P||s_Q||Q|^{1/2}}{\|\chi_P\|^2_{L^{p(\cdot)}}}
\right\}_{Q\subset P}\right\|_{s^{p(\cdot)}}\\
&\le C\left\{\int_P
\bigg(\sum_{Q\subset P}\frac{|s_Q|^2}
{\|\chi_P\|^2_{L^{p(\cdot)}}|P|^{-1}}
\chi_Q(x)\bigg)dx\right\}^{1/2}
\frac{\|\chi_Q\|_{L^{p(\cdot)}}}{\|\chi_P\|_{L^{p(\cdot)}}}
\frac{|P|}{|Q|}\\
\end{align*}

\begin{align*}
&\le C\left\{\int_P
\bigg(\sum_{Q\subset P}\frac{|s_Q|^2}
{\|\chi_P\|^2_{L^{p(\cdot)}}|P|^{-1}}
\chi_Q(x)\bigg)dx\right\}^{1/2}\\
&\le C
\bigg(\sum_{Q\subset P}\frac{|s_Q|^2|Q|}
{\|\chi_P\|^2_{L^{p(\cdot)}}|P|^{-1}}
\bigg)^{1/2}\\
&=\|s\|_{l^2(X,d\sigma)}.
\end{align*}

Thus,
\begin{align*}
\left(\frac{|P|}{\|\chi_P\|^2_{L^{p(\cdot)}}}
\sum_{Q\subset P}|t_Q|^2\right)^{1/2}
&\le\|\ell\|.
\end{align*}

Then let $L\in(H^{p(\cdot)})'$ and define $\ell=L\circ T_\psi$.
By proposition \ref{s3p2}, $\ell\in (s^{p(\cdot)})'$.
Thus, there exists $t=\{t_Q\}_Q\in c^{p(\cdot)}$ such that
$$
\ell(\{s_Q\}_Q)=\sum_Qs_Qt_Q
\quad\mbox{for}\quad f\in s^{p(\cdot)}
$$
and
$\|t\|_{c^{p(\cdot)}}\sim\|\ell\|\le C\|L\|$.
For $f\in c^{p(\cdot)}$,
we have
$$\ell\circ S_\varphi(f)=L\circ T_\psi\circ S_\varphi(f)
=L(f).
$$
Thus, for $f\in\mathcal S_0$
and letting $g=T_\psi(t)=\sum_Qt_Q\psi_Q$,
$$
L(f)=\ell\circ S_\varphi(f)=\sum_Q\left<f,\varphi_Q\right>t_Q
=\left<t,S_\varphi(f)\right>.
$$
Observe that
$\left<g,f\right>=\left<S_\psi(g),S_\varphi(f)\right>$
and
$\left<t,S_\varphi(f)\right>=\left<T_\psi(t),f\right>$.
Then we have,
$$
L(f)=\left<T_\psi(t),f\right>=L_g(f)
\quad\mbox{for}\quad f\in\mathcal S_0.
$$
Therefore, by Proposition \ref{s3p3}
$$
\|g\|_{CMO^{p(\cdot)}}\le C\|t\|_{c^{p(\cdot)}}
\le C\|L\|
$$
and the proof is complete.
\quad $\hfill\Box$

Before we give the proofs to the above two propositions,
we need the
following
equivalent characterizations of $H^{p(\cdot)}$, which,
for the case of inhomogeneous variable Hardy spaces, was
studied in \cite{T18}.

\begin{lemma}\cite{T17AFA}\label{s3l1}\quad Let $p(\cdot)\in LH$. Then for all $f\in \mathcal {S'_\infty}$,
\begin{align*}
\|f\|_{H^{p(\cdot)}}\sim \|\mathcal G(f)\|_{L^{p(\cdot)}}\sim \|\mathcal G^d(f)\|_{L^{p(\cdot)}}.
\end{align*}
\end{lemma}

We now are ready to prove Proposition \ref{s3p2} and
\ref{s3p3}.

\noindent\textit{Proof of Proposition \ref{s3p2}.}
By Lemma \ref{s3l1}, for $f\in H^{p(\cdot)}$,
$$
\|S_\varphi(f)\|_{s^{p(\cdot)}}=
\left\|\bigg\{\sum_Q
|\left<f,\varphi_Q\right>|^2|Q|^{-1}
\chi_Q\bigg\}^{1/2}\right\|_{L^{p(\cdot)}}
\le \|f\|_{H^{p(\cdot)}}.
$$

For $\{s_Q\}\in s^{p(\cdot)}$,
\begin{align*}
\|T_\psi(\{s_Q\})\|_{H^{p(\cdot)}}
&=\left\|\left(\sum_{j' \in \mathbb Z}
\sum_{\mathbf k' \in\mathbb Z^n}|\psi_j\ast (T_\psi(\{s_Q\}))(2^{-j}\mathbf k)|^2\chi_{Q'}\right)^{1/2}\right\|_{L^{p(\cdot)}}\\
&=\left\|\left(\sum_{j' \in \mathbb Z}
\sum_{\mathbf k' \in\mathbb Z^n}\big|\psi_j\ast
\big(\sum_Qs_Q\psi_Q\big)(2^{-j}\mathbf k)\big|^2\chi_{Q'}\right)^{1/2}\right\|_{L^{p(\cdot)}}.
\end{align*}

The rest of the proof
is closely related to \cite[Proposition 2.3]{T17AFA},
that is, it follows the similar routine as the proof of
\cite[Proposition 2.3]{T17AFA}. Namely,
by the almost-orthogonality estimates, the estimate in
\cite[pp. 147, 148]{FJ2}, H\"older's inequality and the
Fefferman-Stein vector-valued maximal function inequality
in Proposition \ref{s2p2}, we get
$$
\|T_\psi(\{s_Q\})\|_{H^{p(\cdot)}}\le C\|\{s_Q\}\|_{s^{p(\cdot)}}
$$

Finally, it is easy to check that from the discrete Calder\'on identity
introduced by Frazier and
Jawerth in Lemma \ref{s2l1},
$T_\psi\circ S_\varphi$ is the identity on
$H^{p(\cdot)}$.
\quad $\hfill\Box$

\noindent\textit{Proof of Proposition \ref{s3p3}.}
For any $g\in CMO^{p(\cdot)}$,
applying Corollary \ref{s2c1} yields
\begin{align*}
\|\{S_\varphi(g)\}\|_{c^p}
&=\sup_{P}\left\{
\frac{|P|}{\|\chi_P\|^2_{L^{p(\cdot)}}}
\sum_{Q\subset P}|\left<g,\varphi_Q\right>|^2
\right\}^{1/2}\\
&=
\sup_{P}\left\{\frac{|P|}{\|\chi_P\|^2_{p(\cdot)}}
\sum_{j=-\log_2\ell(P)}^\infty\sum_{\substack{Q\subset P\\
\ell(Q)=2^{-j-N}}}
|\varphi_j\ast f(x_Q)|^2|Q|\right\}^{1/2}\\
&\le C\|f\|_{CMO^{p(\cdot)}}.
\end{align*}

For $\{s_{Q'}\}\in c^{p(\cdot)}$,
\begin{align*}
&\|T_\psi(\{s_{Q'}\})\|_{CMO^{p(\cdot)}}\\
&=\sup_{P}\left\{\frac{|P|}{\|\chi_P\|^2_{p(\cdot)}}
\int_{\mathbb R^n}\sum_{Q\subset
P}|Q|^{-1}|\left<T_\psi(\{s_{Q'}\}),
\psi_Q\right>|^2\chi_{Q}(x)dx\right\}^{1/2}\\
&=\sup_{P}\left\{\frac{|P|}{\|\chi_P\|^2_{p(\cdot)}}
\int_{\mathbb R^n}\sum_{Q\subset
P}|Q|^{-1}\left|\bigg<\sum_{Q'}s_{Q'}\psi_{Q'},
\psi_Q\bigg>\right|^2\chi_{Q}(x)dx\right\}^{1/2}.
\end{align*}

The rest of this proof is similar to
that of Theorem \ref{s2t1}.
A same argument as the proof of Theorem \ref{s2t1},
we have
$$
\|T_\psi(\{s_{Q'}\})\|_{CMO^{p(\cdot)}}
\le C\|\{s_{Q'}\}\|_{c^{p(\cdot)}}.
$$

Finally, by Lemma \ref{s2l1}
we can easily get that from the Calder\'on reproducing formula
$T_\psi\circ S_\varphi$ is the identity operator on
$CMO^{p(\cdot)}$.
\quad $\hfill\Box$

\section{The Equivalence of $CMO^{p(\cdot)}$}
In this section, we will see that Carleson measure spaces with variable exponents $CMO^{p(\cdot)}$, Campanato space with variable exponent
$\mathfrak{L}_{q,p(\cdot),d}$ and H\"older-Zygmund spaces with variable exponents $\mathcal H_d^{p(\cdot)}$ coincide as sets and the corresponding norms are equivalent.

We first
recall some definitions and lemmas below in \cite{NS}.
Recall that the definition of atomic Hardy space with variable
exponent $H_{atom}^{p(\cdot),q}$.
Let $1<q\le\infty$ and $p(\cdot)\in \mathcal P^0\cap LH$. The function space
$H_{atom}^{p(\cdot),q}$ is defined to be the set of all distributions $f\in
\mathcal S'$ which can be written as
$f=\sum_j\lambda_ja_j$
in $\mathcal S'$, where $\{a_j, Q_j\}\subset \mathcal A(p(\cdot),q)$
with the quantities
\begin{equation*}
  \mathcal{A}(\{\lambda_j\}_{j=1}^\infty,\{Q_j\}_{j=1}^\infty)
  <\infty.
\end{equation*}
One define
$$
\|f\|_{H_{atom}^{p(\cdot),q}}\equiv\mathcal{A}(\{\lambda_j\}_{j=1}^\infty,\{Q_j\}_{j=1}^\infty).
$$

Let $q\gg1$ and $p(\cdot)\in \mathcal P^0\cap LH$.
It is well known that
$$
H^{p(\cdot)}=H^{p(\cdot),q}_{atom}.
$$

We also recall the notion of the Campanato space with variable exponent
$\mathfrak{L}_{q,p(\cdot),s}$.
Write that $\mathcal P^s$ is the set of all polynomials having degree at most $d$.

\begin{definition}\label{s4d1}
Let $p(\cdot)\in\mathcal P$, $d$ be a nonnegative integer and $1\le q<\infty$.
Then the Campanato space with variable exponent
$\mathfrak{L}_{q,p(\cdot),d}$ is defined to be the set of all $f\in L_{loc}^q$
such that
$$
\|f\|_{\mathfrak{L}_{q,p(\cdot),d}}=
\sup_{Q\subset \mathbb R^n}\frac{|Q|}{\|\chi_Q\|_{L^{p(\cdot)}}}
\left[\frac{1}{|Q|}\int_Q|f(x)-P_Q^df(x)|^qdx\right]^{\frac{1}{q}}
<\infty,
$$
where $P_Q^df$ denotes the unique polynomial
$P\in \mathcal P^d$ such that, for all $h\in \mathcal P^d$,
$
\int_Q[f(x)-P(x)]h(x)dx=0.
$
\end{definition}

Let $L^{q,d}_{comp}$ be all the set of all $L^q-$functions with
compact support. For a nonnegative integer $d$, let
$$
L^{q,d}_{comp}=\left\{f\in L^q_{comp}:
\int_{\mathbb R^n}f(x)x^\alpha dx=0,\;
|\alpha|\le d\right\}.
$$

\begin{lemma}\label{s4l1}
Suppose that $p(\cdot)\in LH$, $0<p^-\le p^+\leq1$,
$q>p^+$ and $d$ is a nonnegative integer such that $d\in(n/p^--n-1,\infty)$.
The dual of $H_{atom}^{p(\cdot),q}$,
denoted by $(H_{atom}^{p(\cdot),q})'$
is $\mathfrak{L}_{q',p(\cdot),d}$ in the following sense.\\
\noindent (1) For any $b\in \mathfrak{L}_{q',p(\cdot),d}$, the linear functional $L_b:=\int_{\mathbb R^n}b(x)dx$,
defined initially on $L^{q,d}_{comp}$,
has a  bounded extension
on $H_{atom}^{p(\cdot),q}$ with $\|L_b\|\le
C\|g\|_{\mathfrak{L}_{q',p(\cdot),d}}$.

\noindent (2) Conversely, every continuous linear functional $L$ on $H_{atom}^{p(\cdot),q}$
satisfies $L=L_b$ for some $b\in \mathfrak{L}_{q',p(\cdot),d}$ with $\|b\|_{\mathfrak{L}_{q',p(\cdot),d}}\le C\|L\|$.
\end{lemma}

Define $\Delta_h^k$ to be a difference operator, which is defined inductively
by
$$
\Delta_h^1f=\Delta_hf\equiv f(\cdot+h)-f,
\quad\quad
\Delta_h^k\equiv\Delta_h^1\circ\Delta_h^{k-1},
\quad k\ge 2.
$$

\begin{definition}\label{s4d2}
Let $p(\cdot)\in\mathcal P$, and $d\in\mathbb N\cup\{0\}$.
Then the H\"older-Zygmund spaces with variable exponents,
$\mathcal H_d^{p(\cdot)}$, is defined to be the set of all
continuous functions
$f$
such that
$$
\|f\|_{\mathcal H_d^{p(\cdot)}}=
\sup_{x\in\mathbb R^n,h\neq0}\frac{|Q|}{\|\chi_Q\|_{L^{p(\cdot)}}}
\left|\Delta_h^{d+1}f(x)\right|
<\infty,
$$
where $Q=Q(x,|h|)$.
\end{definition}
Note that we still use
$\mathcal H_d^{p(\cdot)}$ to denote the above function space modulo the
polynomials of degree $d$.

\begin{lemma}\label{s4l2}
Suppose that $p(\cdot)\in LH$, $0<p^-\le p^+\leq1$.
Then the function spaces $\mathcal H_d^{p(\cdot)}$
and $\mathfrak{L}_{q,p(\cdot),d}$ are isomorphic
in the following sense.\\
1. For any $f\in \mathcal H_d^{p(\cdot)}$
we have $\|f\|_{\mathfrak{L}_{q,p(\cdot),d}}
\le C\|f\|_{\mathcal H_d^{p(\cdot)}}$.\\
2. Any element in $\mathfrak{L}_{q,p(\cdot),d}$
has a continuous representative. Moreover, whenever
continuous functions $f\in\mathfrak{L}_{q,p(\cdot),d}$,
then $f\in\mathcal H_d^{p(\cdot)}$
and we have $
\|f\|_{\mathcal H_d^{p(\cdot)}}
\le
C\|f\|_{\mathfrak{L}_{q,p(\cdot),d}}$.
\end{lemma}

Now we state the main result in this section.
\begin{theorem}\label{s4t1}
Suppose that $p(\cdot)\in LH$, $0<p^-\le p^+\leq1$, $1<q<\infty$ and
$d$ is a nonnegative integer such that $d\in(n/p^--n-1,\infty)$.
Then Carleson measure spaces with variable exponents $CMO^{p(\cdot)}$,
Campanato space with variable exponent $\mathfrak{L}_{q,p(\cdot),d}$
and H\"older-Zygmund spaces with variable exponents
$\mathcal H_d^{p(\cdot)}$ coincide as sets and
$$\|f\|_{CMO^{p(\cdot)}}\sim\|f\|_{\mathfrak{L}_{q,p(\cdot),d}}
\sim\|f\|_{\mathcal H_d^{p(\cdot)}}.$$
\end{theorem}

\begin{proof}
Applying Theorem \ref{s3t1} and Lemma \ref{s4l1} yields
$$(H_{atom}^{p(\cdot),q})'=\mathfrak{L}_{q',p(\cdot),d},\quad
(H^{p(\cdot)})'=CMO^{p(\cdot)}.$$
Assume that $q\ge1$ is sufficiently large
and $p(\cdot)\in \mathcal P^0\cap LH$.
Then by \cite[Theorem 4.6]{NS} we have
$$
H^{p(\cdot)}=H^{p(\cdot),q}_{atom},
$$
with
$\|f\|_{H^{p(\cdot)}}\sim\|f\|_{H^{p(\cdot),q}_{atom}}$.
For any given $f\in \mathfrak{L}_{q',p(\cdot),d}$,
we see that $f$ is a linear and continuous on
$H_{atom}^{p(\cdot),q}$.
That is, $\|x_n-x\|_{H_{atom}^{p(\cdot),q}}\rightarrow 0$
implies $|f(x_n)-f(x)|\rightarrow 0$.
On the other hand, for any $\|x_n-x\|_{H^{p(\cdot)}}\rightarrow 0$,
we have $\|x_n-x\|_{H^{p(\cdot)}}\sim
\|x_n-x\|_{H^{p(\cdot),q}_{atom}}\rightarrow 0$.
Then we have $|f(x_n)-f(x)|\rightarrow 0$.
So $f$ is also a linear and continuous on
$H^{p(\cdot)}$ and $f\in CMO^{p(\cdot)}$.
Therefore,
$$\mathfrak{L}_{q',p(\cdot),d}\subset CMO^{p(\cdot)}.$$
Moreover,
\begin{align*}
\|f\|_{CMO^{p(\cdot)}}&\sim
\|L_f\|_{(H^{p(\cdot)})'}
=\sup_{\|f\|_{H^{p(\cdot)}}\le 1}|L_f(h)|\\
&\le \sup_{\|Cf\|_{H_{atom}^{p(\cdot),q}}\le 1}|L_f(h)|
=C\|L_f\|_{(H_{atom}^{p(\cdot),q})'}
\sim\|f\|_{\mathfrak{L}_{q',p(\cdot),d}}.
\end{align*}
Similarly, we also have
$$CMO^{p(\cdot)}\subset \mathfrak{L}_{q',p(\cdot),d}$$
and
\begin{align*}
\|f\|_{\mathfrak{L}_{q',p(\cdot),d}}\le
C\|f\|_{CMO^{p(\cdot)}}.
\end{align*}
Thus, $CMO^{p(\cdot)}$
and $\mathfrak{L}_{q',p(\cdot),d}$
coincide as sets and
\begin{align}\label{s4i1}
\|f\|_{CMO^{p(\cdot)}}\sim\|f\|_{\mathfrak{L}_{q',p(\cdot),d}}.
\end{align}
According to \cite[Corollary 2.22]{ZYL},
for any $1<q<\infty$,
$f\in \mathfrak{L}_{q,p(\cdot),d}$
if and
only if
$f\in \mathfrak{L}_{1,p(\cdot),d}$
and
\begin{equation}
\label{s4i2}
\|f\|_{\mathfrak{L}_{q,p(\cdot),d}}
\sim\|f\|_{\mathfrak{L}_{1,p(\cdot),d}}.
\end{equation}

Furthermore, by Lemma \ref{s4l2},
$\mathcal H_d^{p(\cdot)}$
and $\mathfrak{L}_{q,p(\cdot),d}$ are isomorphic
and
\begin{align}\label{s4i3}
\|f\|_{\mathcal H_d^{p(\cdot)}}
\sim\|f\|_{\mathfrak{L}_{q,p(\cdot),d}}.
\end{align}

Combining both (\ref{s4i1}), (\ref{s4i2}) and (\ref{s4i3}),
we immediately obtain that
$$\|f\|_{CMO^{p(\cdot)}}\sim\|f\|_{\mathfrak{L}_{q,p(\cdot),d}}
\sim\|f\|_{\mathcal H_d^{p(\cdot)}}.$$
This proves Theorem \ref{s4t1}.
\end{proof}

\section{Applications}
In this section
we show that Calder\'{o}n-Zygmund singular integral operators are
bounded on $CMO^{p(\cdot)}$ via
using an argument of weak
density property.
Note that the weak
density property is very useful when
we deal with the bounedness of operators
on Carleson measure type spaces or Lipschitz type spaces
(see \cite{HH,LiW,L,LW,TZ}).
First we recall some necessary notations and definitions.

We say $K(x,y)\in C_c^\infty$ is
a standard kernel, if it is
defined for $x\neq y$, and satisfies the following
weaker version of the differential inequalities:
$$|K(x,y)|\leq C|x-y|^{-n};$$
if $|x-y|\geq2|y-y'|$,
$$|K(x,y)-k(x,y')|\leq C|y-y'|^{\gamma}|x-y|^{-n-\gamma}
$$and if $|x-y|\geq2|x-x'|$,
$$|K(x,y)-k(x',y)|\leq C|x-x'|^{\gamma}|x-y|^{-n-\gamma};
$$where $0<\gamma\leq1$.
We denote $K\in SK(\gamma)$.

When $K\in SK(\gamma)$, if $\varphi,\psi\in C^\infty_c$,
$\mbox{supp}\;\varphi\cap\mbox{supp}\;\psi=\emptyset,$
then
$$
\left<T\varphi,\psi\right>
=\int_{\mathbb R^n\times\mathbb R^n}K(x,y)\varphi(y)\psi(y)dydx.
$$
That is, for any $x\neq\mbox{supp}\;\varphi$
$$
T(\varphi)(x)
=\int_{\mathbb R^n}K(x,y)\varphi(y)dy.
$$

Suppose that $T$ is bounded on $L^{2}$.
The relation between $K$ and $T$ is that
$f\in L^2$ has compact support, then, outside the support of $f$,
the distribution $Tf$ agrees with the function
$$T(f)(x)=\int_{\mathbb{R}^{n}}K(x,y)f(y)dy,\quad x\notin {\rm supp} (f)$$
Then $T$ is Calder\'{o}n-Zygmund operator.
We denote $T\in CZO(\gamma)$.

The adjoint operator $T^\ast$ is defined by
$$
\left<T^\ast f,g\right>=\left<f,Tg\right>,
\quad f,g\in\mathcal S.
$$
It is associated with the standard kernel
$\tilde K(x,y)=K(y,x)$.

Note that
$T\in CZO(\gamma)$ can be extended to a
bounded linear operator on $H^q$ for all $\frac{n}{n+\gamma}<q<\infty$,
when $T^\ast(1)=0$,
and that its adjoint $T^\ast$ also
can be extended to a
bounded linear operator on $H^q$ for all $\frac{n}{n+\gamma}<q<\infty$,
when $T(1)=0$, for a proof, see \cite[Section 10, Theorem 1.1]{DH1}.
By repeating the similar argument to the proof of \cite[Theorem 5.5]{NS},
we can prove the following result:

\begin{proposition}\label{s5p1}
Suppose that $T\in CZO(\gamma)$, $p(\cdot)\in LH$
 and $\frac{n}{n+\gamma}<p^-\le p^+<\infty$. If $T^\ast(1)=0$,
then $T$ is a
bounded linear operator on $H^{p(\cdot)}$.
Similarly,
if $T(1)=0$,
then $T^\ast$ is a
bounded linear operator on $H^{p(\cdot)}$.
\end{proposition}

Our main result in this section is the following theorem.

\begin{theorem}\label{s5t1}
Suppose that $T\in CZO(\gamma)$, $p(\cdot)\in LH$
 and $\frac{n}{n+\gamma}<p^-\le p^+\le 1$.
 If $T(1)=0$,
then $T$ admits a bounded extension from
$CMO^{p(\cdot)}$  to itself.
\end{theorem}

The following proposition
on the weak density property for
$CMO^{p(\cdot)}$  plays a key role in the proof of Theorem \ref{s5t1}.
\begin{proposition}\label{s5p2}
Let $p(\cdot)\in LH$
 and $0<p^-\le p^+\le 1$.
If $f\in CMO^{p(\cdot)}$, then there exist
a sequence $\{f_m\}\in
CMO^{p(\cdot)}\cap L^2$ such that
$f_m$ converges to
$f$ in the distribution sense. Furthermore,
$$\|f_m\|_{CMO^{p(\cdot)}}\le C\|f\|_{CMO^{p(\cdot)}},
\quad \mbox{for}\quad f\in CMO^{p(\cdot)}.$$
\end{proposition}

\begin{proof}
By Lemma \ref{s2l1}, we have the following Caldr\'{o}n identity
$$
f(x)=\sum_{j\in\mathbb Z}\psi_{j}\ast\psi_{j}\ast f(x) \quad\quad \mbox{in}\quad \mathcal{S}'_\infty.
$$

The partial sum of the above series will be denoted by $f_m$ and is given by
$$
f_m(x)=\sum_{|j|\le m}\psi_{j}\ast\psi_{j}\ast f(x).
$$

Then we claim that
$$
\|f_m\|_{L^2}<\infty.
$$
In fact, applying the fact that
$|\psi_{j}\ast f(x)|\le C$
proved in Theorem \ref{s2t1}
yields that
\begin{align*}
\|\psi_{j}\ast\psi_{j}\ast f\|_{L^2}\le C.
\end{align*}

For any $g\in\mathcal{S}_\infty$,
we obtain that
\begin{align*}
|\left<f-f_m,g\right>|&=
\lim_{n\rightarrow\infty}|\left<f_n-f_m,g\right>|\\
&\le\liminf_{n\rightarrow\infty}
\left|\left<\sum_{m<|j|\le n}\psi_{j}\ast\psi_j\ast f, g\right>\right|
\rightarrow 0,  \quad\mbox{as}\quad m\rightarrow 0.
\end{align*}
Thus, $f_m\in L^2$ and converges to $f$ in the distribution sense.

To conclude the proof, note that
$$
\psi_{j}\ast f_m(x)=\sum_{|j'| \le m}\psi_{j}\ast \psi_{j'} \ast\psi_{j'}\ast f(x)
$$

and by Corollary \ref{s2c1}, it follows that
\begin{align*}
\|f_m\|_{CMO^{p(\cdot)}}
\sim&
\sup_{P}\left\{\frac{|P|}{\|\chi_P\|^2_{p(\cdot)}}
\sum_{j=-\log_2\ell(P)}^\infty\sum_{\substack{Q\subset P\\
\ell(Q)=2^{-j-N}}}
|\varphi_j\ast f_m(x_Q)|^2|Q|\right\}^{1/2}.
\end{align*}

Again applying the almost orthogonal estimate and Corollary \ref{s2c1},
for any $j\ge0$ we have
$\|f_m\|_{CMO^{p(\cdot)}}\le
C\|f\|_{CMO^{p(\cdot)}}$.

Therefore, the proof of Proposition \ref{s5p2} is completed.
\end{proof}

Now we prove of Theorem \ref{s5t1}.\\
\noindent\textit{Proof of Theorem \ref{s5t1}.}\quad
First we show that the Calder\'{o}n-Zygmund operator $T$
is a bounded operator on $CMO^{p(\cdot)}\cap L^2$.
Applying Theorem \ref{s3t1} and Proposition \ref{s5p1}
yield that
$$
|\left<Tf,g\right>|=|\left<f,T^\ast g\right>|\le
\|f\|_{CMO^{p(\cdot)}}\|T^\ast g\|_{H^{p(\cdot)}}
\le C\|f\|_{CMO^{p(\cdot)}}\|g\|_{H^{p(\cdot)}}.
$$
That is, for each $f\in CMO^{p(\cdot)}\cap L^2$,
$L_f(g)=\left<Tf,g\right>$ is a continuous linear functional
on $H^{p(\cdot)}\cap L^2$.
Since $H^{p(\cdot)}\cap L^2$ is dense in $H^{p(\cdot)}$,
$L_f$
can be extended to a continuous linear functional
on $H^{p(\cdot)}$ with
$$\|L_f\|\le C\|f\|_{CMO^{p(\cdot)}}.$$

Conversely, by Theorem \ref{s3t1} again, there exists
$h\in CMO^{p(\cdot)}$ such that $\left<Tf,g\right>
=\left<h,g\right>$ for $g\in H^{p(\cdot)}\cap L^2$ with
$$\|h\|_{CMO^{p(\cdot)}}\le C\|L_f\|.$$

Thus,
\begin{align*}
\|Tf\|_{CMO^{p(\cdot)}}
&=\sup_{P}\left\{\frac{|P|}{\|\chi_P\|^2_{p(\cdot)}}
\int_{\mathbb R^n}\sum_{Q\subset
P}|Q|^{-1}|\left<Tf, \psi_Q\right>|^2\chi_{Q}(x)dx\right\}^{1/2}\\
&=\sup_{P}\left\{\frac{|P|}{\|\chi_P\|^2_{p(\cdot)}}
\int_{\mathbb R^n}\sum_{Q\subset
P}|Q|^{-1}|\left<h, \psi_Q\right>|^2\chi_{Q}(x)dx\right\}^{1/2}\\
&=\|h\|_{CMO^{p(\cdot)}}\le C\|L_f\|\le C\|f\|_{CMO^{p(\cdot)}}.
\end{align*}

Next we extend this result to $CMO^{p(\cdot)}$
via an argument of weak density property.
Suppose that $f\in CMO^{p(\cdot)}$.
By Proposition \ref{s5p2}, we can choose a sequence
$\{f_m\}\subset CMO^{p(\cdot)}\cap L^2$
with $$\|f_m\|_{CMO^{p(\cdot)}}\le
C\|f\|_{CMO^{p(\cdot)}}$$
such that
$f_m$ converges to
$f$ in the distribution sense.
Therefore,
for $f\in CMO^{p(\cdot)}$,
define
$$
\left<Tf,g\right>=\lim_{m\rightarrow\infty}
\left<Tf_m,g\right>,\quad\mbox{for}\quad g\in H^{p(\cdot)}
\cap L^2.
$$
In fact, we have
$\left<T(f_i-f_j),g\right>=\left<f_i-f_j,T^\ast(g)\right>$,
where $f_i-f_j$ and $g$ belong to $L^2$.
By Proposition \ref{s5p1}, we have $T^\ast\in H^{(p(\cdot))}\cap L^2$.
Applying Proposition \ref{s5p2} again, we get that
$$
\left<T(f_i-f_j),g\right>
=\left<f_i-f_j,T^\ast(g)\right>\rightarrow 0
$$
as $j,\;k\rightarrow \infty$.
Therefore, $Tf$ is well defined
and
$$
\left<Tf,g\right>=\lim_{m\rightarrow\infty}\left<Tf_m,g\right>
$$
for any $g\in H^{p(\cdot)}\cap L^2$ and $f_m\in CMO^{p(\cdot)}\cap L^2$.
Then by Fatou's Lemma, for each dyadic cube $P$ in $\mathbb R^n$,
\begin{align*}
&\left\{\frac{|P|}{\|\chi_P\|^2_{p(\cdot)}}
\int_{\mathbb R^n}\sum_{Q\subset
P}|Q|^{-1}|\left<Tf, \psi_Q\right>|^2\chi_{Q}(x)dx\right\}^{1/2}\\
&\le \liminf_{m\rightarrow\infty}
\left\{\frac{|P|}{\|\chi_P\|^2_{p(\cdot)}}
\int_{\mathbb R^n}\sum_{Q\subset
P}|Q|^{-1}|\left<Tf_m, \psi_Q\right>|^2\chi_{Q}(x)dx\right\}^{1/2}.
\end{align*}

Hence,
$$
\|Tf\|_{CMO^{p(\cdot)}}
\le \liminf_{m\rightarrow\infty}
\|Tf_m\|_{CMO^{p(\cdot)}}
\le C\|f_m\|_{CMO^{p(\cdot)}}
\le C\|f\|_{CMO^{p(\cdot)}}.
$$
This completes this proof.
$\hfill\Box$

As a corollary, we obtain that the convolution type
Calder\'on-Zygmund singular integrals
are bounded on $CMO^{p(\cdot)}$.

\begin{corollary}\label{s5c1}
Assume that $p(\cdot)\in LH$
and $0<p^-\le p^+\le 1$.
Let $k\in\mathcal S$ and
$$
\sup_{x\in\mathbb R^n}
|x|^{n+m}|\nabla^mk(x)|<\infty\quad(m\in\mathbb N\cup\{0\}).
$$
Define a convolution operator $\mathcal {T}$ by
$$\mathcal Tf(x)=k\ast f(x)\quad (f\in L^2).$$
Then $\mathcal T$ admits a bounded extension from
$CMO^{p(\cdot)}$  to itself.
\end{corollary}

Combining both Theorems \ref{s4t1} and Theorem \ref{s5t1},
we also have the following corollary.

\begin{corollary}\label{s5c2}
Suppose that $T\in CZO(\gamma)$.
Let $p(\cdot)\in LH$,
 $\frac{n}{n+\gamma}<p^-\le p^+\leq1$, $1<q<\infty$ and
$d$ is a nonnegative integer such that $d\in(n/p^--n-1,\infty)$.
 If $T(1)=0$,
then $T$ can be extended to a bounded operator on
$\mathfrak{L}_{q,p(\cdot),d}$  or
$\mathcal H_d^{p(\cdot)}$.
\end{corollary}

\section*{Acknowledgements}

The project is sponsored by Natural Science Foundation of Jiangsu Province of China (grant no. BK20180734), Natural Science Research of Jiangsu Higher Education Institutions of China (grant no. 18KJB110022) and Nanjing University of Posts and Telecommunications Science Foundation (grant no. NY217151).

\bibliographystyle{amsplain}

\end{document}